\documentclass[12pt,a4paper]{amsart}

\usepackage{amsfonts, amsmath, amssymb, amsthm, amscd, hyperref}

\usepackage{a4wide}

\newtheorem{thm}{Theorem}
\newtheorem*{thm*}{Theorem}
\newtheorem{lem}{Lemma}

\newtheorem{con}{Conjecture}
\newtheorem{fcon}[con]{(False) Conjecture}
\newtheorem{cor}[thm]{Corollary}

\theoremstyle{definition}
\newtheorem{defn}{Definition}
\newtheorem{prob}{Problem}

\theoremstyle{remark}
\newtheorem*{rem}{Remark}

\newcommand{\pt}{\mathrm{pt}}

\renewcommand{\epsilon}{\varepsilon}

\begin{document}

\title[Dvoretzky type theorems for multivariate polynomials\dots]{Dvoretzky type theorems for multivariate polynomials and sections of convex bodies}

\author{V.L.~Dol'nikov}
\thanks{The research of V.L.~Dol'nikov is supported by the Russian Foundation for Basic Research grant 10-01-00096.}
\email{dolnikov@uniyar.ac.ru}

\address{
Vladimir Dol'nikov, Department of Algebra, Yaroslavl' State University, Sovetskaya st. 14, Yaroslavl', Russia 150000}

\author{R.N.~Karasev}
\thanks{The research of R.N.~Karasev is supported by the Dynasty Foundation, the President's of Russian Federation grant MK-113.2010.1, the Russian Foundation for Basic Research grants 10-01-00096 and 10-01-00139, the Federal Program ``Scientific and scientific-pedagogical staff of innovative Russia'' 2009--2013}

\email{r\_n\_karasev@mail.ru}
\address{
Roman Karasev, Dept. of Mathematics, Moscow Institute of Physics
and Technology, Institutskiy per. 9, Dolgoprudny, Russia 141700}


\begin{abstract}
In this paper we prove the Gromov--Milman conjecture (the Dvoretzky type theorem) for homogeneous polynomials on $\mathbb R^n$, and improve bounds on the number $n(d,k)$ in the analogous conjecture for odd degrees $d$ (this case is known as the Birch theorem) and complex polynomials. 

We also consider a stronger conjecture on the homogeneous polynomial fields in the canonical bundle over real and complex Grassmannians. This conjecture is much stronger and false in general, but it is proved in the cases of $d=2$ (for $k$'s of certain type), odd $d$, and the complex Grassmannian (for odd and even $d$ and any $k$). Corollaries for the John ellipsoid of projections or sections of a convex body are deduced from the case $d=2$ of the polynomial field conjecture.
\end{abstract}

\maketitle

\section{Introduction}

The following theorem was conjectured in~\cite{mil1988} (see also~\cite{mil1992}), it is known as the Gromov--Milman conjecture. This theorem resembles the famous theorem of Dvoretzky~\cite{dvor1961} on near-elliptical sections of convex bodies. It considers polynomials instead of convex bodies, and unlike the Dvoretzky theorem, it gives strict ``roundness'' rather than approximate ``roundness''.

\begin{thm}
\label{ramsey-pol}
For an even positive integer $d$ and a positive integer $k$ there exists $n(d, k)$ such that for any homogeneous polynomial $f$ of degree $d$ on $\mathbb R^n$, where $n\ge n(d,k)$, there exists a linear $k$-subspace $V\subseteq \mathbb R^n$ such that $f|_V$ is proportional to the $d/2$-th power of the standard quadratic form
$$
Q = x_1^2+x_2^2+\dots+x_n^2.
$$
\end{thm}

\begin{rem}
Actually, the conjecture in~\cite{mil1988} was stated in a slightly different way: the restriction $f|_V$ was required to be proportional to the $d/2$-th power of \emph{some} quadratic form. But a straightforward argument (using the diagonal form in an orthonormal basis) shows that $n(2, k) = 2k-1$, i.e. any quadratic form on $\mathbb R^{2k-1}$ is proportional to the standard form on a subspace of dimension $k$. Hence these two versions are equivalent modulo the precise values of $n(d,k)$, and this equivalence is essentially used in the proofs.
\end{rem}
\begin{rem}
In~\cite{mil1988} it was also conjectured that $n(d,k)$ is of order $k^d$. We do not have results of this kind here because we use a topological Borsuk--Ulam type theorem without explicit bound, see Section~\ref{bu-p-sec} and the remark at the end of Section~\ref{proof-final}.
\end{rem}

Besides the trivial case $d=2$, there were other partial results in this conjecture. Theorem~\ref{ramsey-pol} was proved in~\cite{mil1988, mak1990,mak2003} (the essential idea goes back to M.~Gromov) for $k=2$ by topological methods (actually, the stronger Conjecture~\ref{grass-pol} was proved for $k=2$), and with good bounds for $n(d,2)$. In case of special polynomials of the form $f=x_1^d+x_2^d+\dots+x_n^d$ this theorem was proved in~\cite{sch1913}, see also~\cite{mil1988} for a short proof with the averaging trick. If we let $d$ be odd, this theorem is known as the Birch theorem and holds in a stronger form with good estimates on $n(d,k)$, see~\cite{bir1957,arha2006} and Theorems~\ref{grass-pol-odd} and~\ref{odd-maps} below. In this paper we combine the topological technique with the averaging method of~\cite{mil1988} to prove Theorem~\ref{ramsey-pol}. 

Let us state a more general conjecture, that would imply Theorem~\ref{ramsey-pol}, if it were true.

\begin{defn}
Denote $G_n^k$ the Grassmannian of linear $k$-subspaces in $\mathbb R^n$, denote by $\gamma_n^k : E(\gamma_n^k)\to G_n^k$ its canonical vector bundle.
\end{defn}

\begin{defn}
For a vector bundle $\xi : E(\xi)\to X$ denote $\Sigma^d(\xi)$ its fiberwise symmetric $d$-th power. We consider every vector bundle $\xi$ along with some Riemannian metric on its fibers, i.e. a nonzero section $Q(\xi)$ of $\Sigma^2(\xi)$ inducing a positive quadratic form on every fiber.
\end{defn}

\begin{fcon}
\label{grass-pol}
Suppose $d$ and $k$ are even positive integers. Then there exists $n(d, k)$ such that for every section of the bundle $\Sigma^d(\gamma_n^k)$ over $G_n^k$ with $n\ge n(d,k)$, there exists a subspace $V\in G_n^k$ such that this section is a multiple of $(Q(\gamma_n^k))^{d/2}$ over $V$.
\end{fcon}

This conjecture would imply Theorem~\ref{ramsey-pol}, because every polynomial of degree $d$ defines a section of $\Sigma^d(\gamma_n^k)$ tautologically.

Unfortunately, there already exist some negative results on Conjecture~\ref{grass-pol}. It is shown in~\cite[Ch.~IV, \S~1 (A)]{hsiang1975} (with reference to~\cite{hsiang1967}) that this conjecture fails for odd $k$. The counterexample is for $d=2$ and the oriented Grassmannian (the space $BSO(k)$), but it seems like the case of the Grassmannian $G_\infty^k = BO(k)$ is handled in the same way. The counterexamples for even $d>2$ are obtained by taking the $d/2$-th power of the counterexample for $d=2$. In~\cite{buivta2009} a counterexample to Conjecture~\ref{grass-pol} is given for $k=4$ and $d\ge 4$.

Of course, these counterexamples do not give a counterexample to the original Theorem~\ref{ramsey-pol}. Moreover, it would be sufficient to prove Conjecture~\ref{grass-pol} for some infinite sequence of $k$'s in order to deduce Theorem~\ref{ramsey-pol} for all $k$'s. 

As noted, Conjecture~\ref{grass-pol} is known to be true for $k=2$, see~\cite{mil1988,mak1990,mak2003}. Here we prove another particular case.

\begin{thm}
\label{grass-quad}
Conjecture~\ref{grass-pol} is true for $d=2$ and $k=2p^\alpha$ for a prime $p$. In the first nontrivial case we have a particular estimate
$$
n(2, 4) \le 12.
$$
\end{thm}

Note that this theorem does not add anything new to Theorem~\ref{ramsey-pol} (the quadratic forms are not interesting there), but it has some applications to sections and projections of convex bodies, given in Section~\ref{proj-sec}. Theorem~\ref{grass-quad} has the following generalization for several sections.

\begin{thm}
\label{grass-quad-mult}
Suppose $k=2p^\alpha$ for a prime $p$, $m$ is a positive integer. Then there exists $n(2, k, m)$ such that for every $m$ sections $s_1, \ldots, s_m$ of the bundle $\Sigma^d(\gamma_n^k)$ over $G_n^k$ with $n\ge n(2,k,m)$, there exists a subspace $V\in G_n^k$ such that all the sections $s_i$ are multiples of $(Q(\gamma_n^k))^{d/2}$ over $V$.
\end{thm}

The topological proof of Theorems~\ref{grass-quad} and \ref{grass-quad-mult} cannot be applied directly to the cases of Conjecture~\ref{grass-pol} with $d\ge 4$. But returning to Theorem~\ref{ramsey-pol} for multivariate polynomials, we shall see that the similar topological technique is essentially used, along with the averaging trick and some combinatorics.

Now let us turn to the case of odd $d$ in Theorem~\ref{ramsey-pol} and Conjecture~\ref{grass-pol}. This version of Theorem~\ref{ramsey-pol} was known even before the formulation of this conjecture for even-degree polynomials in~\cite{mil1988}. In~\cite{bir1957} it was shown to be true in an (obviously) stronger from, i.e. $f=0$ on a $k$-dimensional subspace. In~\cite{arha2006} the bound on $n(d,k)$ was improved. We are going to improve the bound in~\cite{arha2006}, at least by a factor of $k!$. Moreover, we prove the corresponding result about the sections over the Grassmannian. The topological technique in our proof was also used in~\cite{buivta2009} to study Dvoretzky type theorems over Grassmannians.

\begin{thm}
\label{grass-pol-odd}
Suppose $d$ and $k$ are positive integers, $d$ being odd. Then there exists 
$$
n(d, k) = k + \binom{d+k-1}{d}
$$ 
such that every section of the bundle $\Sigma^d(\gamma_n^k)$ over $G_n^k$ (with $n\ge n(d,k)$) has a zero.
\end{thm}

Simple dimension considerations show that $n(d,k)$ cannot be made less than
$$
k + \frac{1}{k} \binom{d+k-1}{d}
$$
in this theorem. We may conclude that the bound in Theorem~\ref{grass-pol-odd} is quite satisfactory, but still may be improved.
 
The topological approach with Grassmannians also allows us to prove the following version of Theorem~\ref{ramsey-pol} for odd polynomial maps. Of course, the corresponding version of Conjecture~\ref{grass-pol} is also true, but its statement would be too complicated, so we formulate the statement without the Grassmannian and bundles here.

\begin{thm}
\label{odd-maps}
Suppose $d,k,m$ are positive integers, $d$ being odd. Then there exists 
$$
n(d, k, m) = k + m\sum_{1\le \delta\le d,\ \delta\equiv 1\mod 2} \binom{\delta+k-1}{\delta}
$$ 
with the following property. Suppose $f: \mathbb R^n\to \mathbb R^m$ is an odd polynomial map, such that $n\ge n(d,k,m)$, and the coordinate functions of $f$ have degrees $\le d$. Then $f$ maps some $k$-dimensional linear subspace $L\subset \mathbb R^n$ to zero.
\end{thm}

Similar theorems are also true for complex polynomials, Grassmannians, and bundles. In this case the degree does not have to be odd, it can be arbitrary. The following result is the complex analogue of Conjecture~\ref{grass-pol}.

\begin{thm}
\label{grass-pol-comp}
Suppose $d$ and $k$ are positive integers. Then there exists 
$$
n(d, k) = k + \binom{d+k-1}{d}
$$ 
such that every section of the bundle $\Sigma^d(\mathbb C\gamma_n^k)$ over $\mathbb CG_n^k$ (with $n\ge n(d,k)$) has a zero.
\end{thm}

The following result is the unified stronger analogue of Theorems~\ref{ramsey-pol} and \ref{odd-maps} (even and odd degrees) over the complex field.

\begin{thm}
\label{compl-maps}
Suppose $d,k,m$ are positive integers. Then there exists 
$$
n(d, k, m) = k + m \binom{d+k}{d}
$$ 
with the following property. Suppose $f: \mathbb C^n\to \mathbb C^m$ is a polynomial map, such that $n\ge n(d,k,m)$, and the coordinate functions of $f$ have degrees $\le d$. Then $f$ maps some $k$-dimensional linear subspace $L\subset \mathbb C^n$ to zero.
\end{thm}

The authors thank Vitali Milman for useful discussion and comments, Ilya Bogdanov and Dima Faifman, who read the paper to verify the reasoning, and the referee for numerous useful remarks.

\section{Proof of theorems on odd and complex polynomials}
\label{euler-sec}

We start with proofs of the theorems concerning odd and complex polynomials, because their proofs are simple and give a good idea of the topological machinery. The reader can find standard topological facts about characteristic classes of vector bundles in the textbooks~\cite{hatcher2002,milsta1974,mishch1998}, if  needed.

First, consider the infinite Grassmannian $G_\infty^k$ and the canonical bundle $\gamma_\infty^k$ over it. The cohomology $H^*(G_\infty^k, \mathbb F_2)$ is a subalgebra of $\mathbb F_2[t_1,\ldots, t_k]$, consisting of symmetrical polynomials, the Stiefel-Whitney class of $\gamma_\infty^k$ is
$$
w(\gamma_\infty^k) = \prod_{i=1}^k (1+t_i).
$$
Hence, the Stiefel-Whitney class of $\Sigma^d(\gamma_\infty^k)$ is 
$$
w(\Sigma^d(\gamma_\infty^k)) = \prod_{i_1,\ldots, i_k\ge 0}^{i_1+\dots + i_k=d} (1 + i_1t_i + \dots +i_kt_k).
$$
Since $d$ is odd, then none of the expressions $i_1t_i + \dots i_kt_k$ is zero $\mod 2$, and we obtain that the topmost Stiefel-Whitney class of $\Sigma^d(\gamma_\infty^k)$ of dimension $\binom{d+k-1}{d}$ is nonzero. By the standard reasoning this means that $\Sigma^d(\gamma_\infty^k)$ cannot have a section without zeros. 

Now we have to go back to finite Grassmannians. The kernel of the natural map $H^*(G_\infty^k, \mathbb F_2)\to H^*(G_n^k, \mathbb F_2)$ is generated by the dual Stiefel-Whitney classes of $\gamma_\infty^k$ of dimensions $>n-k$. If $n\ge k + \binom{d+k-1}{d}$ the topmost Stiefel-Whitney class of $\gamma_n^k$ turns out to be nonzero from the dimension considerations.

The proof of Theorem~\ref{odd-maps} proceeds in the same way, considering the bundle 
$$
\bigoplus_{1\le \delta\le d,\ \delta\equiv 1\mod 2} \Sigma^\delta(\gamma),
$$
and taking its $m$-fold Whitney power. Obviously, the topmost Stiefel-Whitney class of the resulting bundle over $G_\infty^k$ is nonzero and has dimension 
$$
n_0=m\sum_{1\le \delta\le d,\ \delta\equiv 1\mod 2} \binom{\delta+k-1}{\delta}.
$$
Then we can pass to the finite Grassmannian $G_{n_0+k}^k$ as above.

The proof of Theorems~\ref{grass-pol-comp} and~\ref{compl-maps} is the same with Chern classes in $H^*(\mathbb CG_n^k, \mathbb Z)$ instead of Stiefel-Whitney classes. In this case the topmost Chern class of $\Sigma^d(\mathbb C\gamma_\infty^k)$ is always nonzero. Besides, in Theorem~\ref{compl-maps} we use the well-known formula
$$
\sum_{\delta=0}^d \binom{\delta+k-1}{k-1} = \binom{d+k}{k}.
$$ 

\section{Borsuk--Ulam property for $p$-toral groups}
\label{bu-p-sec}

Before proving Theorems~\ref{ramsey-pol} and~\ref{grass-quad} we need to consider the following Borsuk--Ulam type problem, see the books~\cite{hsiang1975,bart1993} for facts and definitions, concerning the continuous group actions. By $EG$ we denote a homotopy trivial $G$-$CW$-complex with free action of $G$, here it is sufficient to consider $EG$ as the infinite join $G*G*G*\dots$. All maps and sections of vector bundles are assumed to be continuous.

\begin{prob}
\label{equiv-sect}
Suppose $G$ is a compact Lie group, $V$ its representation. There are three equivalent problems:

i) Determine whether the vector bundle 
$$
EG\times V\to EG
$$ 
has a $G$-equivariant nonzero section.

ii) Determine whether there exists a $G$-equivariant map $f : EG\to S(V)$ to the sphere space of the bundle.

iii) Determine whether the vector bundle 
$$
(EG\times V)/G\to BG = EG/G
$$
has a nonzero section.
\end{prob}

In this section it is convenient to use the statement of this problem in version (ii), with an equivariant map $f : EG\to S(V)$. This version is obviously a generalization of the Borsuk--Ulam theorem.  

Of course, if the representation $V$ has a nonzero fixed point set $V^G$, the map $f$ obviously exists, we can map $EG$ to any point in $V^G\setminus \{0\}$. The following result from~\cite{bcp1991,bart1992,bart1993,clm2000} gives an inverse statement for a special class of groups.

\begin{defn}
Suppose 
$$
0\to T\to G\to F\to 0
$$
is an exact sequence of groups, where $T$ is a torus, $F$ is a finite $p$-group. In this case $G$ is called \emph{$p$-toral}.
\end{defn}

\begin{lem}
\label{bu-p-toral} Suppose $G$ is a $p$-toral group and $V$ its representation. Then the image of any equivariant map $f : EG\to V$ intersects $V^G$. 

The following also holds.  There exists $n(G, V)$ such that if a free $G$-space $X$ is $(n-1)$-connected ($n\ge n(G, V)$) then the image of an equivariant map $f : X\to V$ intersects $V^G$. 
\end{lem}

Let us outline briefly the proof of this lemma following~\cite[Proposition~15]{clm2000}. First, it is enough to consider finite $p$-groups, because $p$-toral groups can be ``approximated'' by $p$-groups. 

Consider the first part: the domain is $EG$. The representation $U=V/V^G$ has the only $G$-invariant vector $0$. Denote the composition of $f$ with the projection onto $U$ by $h$. If the image of $f$ does not intersect $V^G$ then the image of $h$ does not contain $0$. Therefore we obtain a $G$-equivariant map $h: EG\to S(U)$, where $S(U)$ denotes the sphere of the representation $U$.

For any $G$-space $X$ there exists a natural map $\pi_X : \Pi_G^0(\pt)\to \Pi_G^0(X)$, where $\pt$ is a point with trivial $G$-action and $\Pi_G^0$ denotes the stable $G$-equivariant $0$-dimensional cohomotopy. Note also that there is a natural isomophism $A(G)\to \Pi_G^0(\pt)$, where $A(G)$ is the Burnside ring of $G$, see~\cite{carl1984}. Now if we take $S(U)$ as $X$ then the natural map $A(G)\to \Pi_G^0(S(U))$ must have a kernel (see explanations in~\cite{clm2000}). And if we take $EG$ as $X$, then the natural map $A(G)\to \Pi_G^0(EG)$ is injective, see~\cite{carl1984}. Therefore there cannot exist a $G$-equivariant map $h$ from $EG$ to $S(U)$.

Now let us turn to the existence of $n(G,V)$. Following~\cite[Proposition~15]{clm2000} note that by definition 
$$
\Pi_G^0(EG) = \lim_{\leftarrow} \Pi_G^0(EG_n),
$$
where $EG_n$ are $n$-skeleta of $EG$. This means that the map $h : EG_n \to S(U)$ is also prohibited for large enough $n$, say $n\ge n(G,V)$. Assume existence of a $G$-equivariant map $h_X: X\to S(U)$ for some $(n-1)$-connected $X$. From the $G$-equivariant obstruction theory and high connectivity of $X$ there exists a $G$-equivariant map $g : EG_n\to X$, composing them $h=h_X\circ g$ we obtain a contradiction.

\section{The rational obstructions to nonzero sections of vector bundles}
\label{rational-sect}

In order to give a particular bound in Theorem~\ref{grass-quad} for $k=4$, we also need the following expression of the rational obstruction to a nonzero section of some vector bundle $\xi : E(\xi)\to X$. Suppose that $\xi$ is oriented. In case $\dim\xi$ is even the first obstruction is the Euler class $e(\xi)$. Consider the case $\dim\xi = 2m+1$. In this case the rational Euler class is zero, but if $\xi$ has a nonzero section, then we have $\xi=\eta\oplus\varepsilon$, where 
$$
\varepsilon : X\times \mathbb R\to X
$$
is the trivial bundle. The bundle $\eta$ is naturally oriented, and we have (we index the Pontryagin classes by their dimension) 
$$
p_{4m}(\eta) = e(\eta)^2.
$$
Since $p_{4m}(\xi) = p_{4m}(\eta)$ we see that the nonexistence of the square root $\sqrt{p_{4m}(\xi)}$ in $H^*(X,\mathbb Z)$ is an obstruction to a nonzero section of $\xi$. 

Considering the fiberwise Postnikov tower for the sphere bundle $S(\xi)$ it can be shown that this is the only rational obstruction for a nonzero section of $\xi$, but we do not need this fact here.

\section{Proof of Theorems~\ref{grass-quad} and \ref{grass-quad-mult}}

First, let us prove the existence of $n(2,k,m)$ in Theorem~\ref{grass-quad-mult} using Lemma~\ref{bu-p-toral}. We can consider $G_\infty^k$, the existence of $n(2,k,m)$ follows from the obstruction theory as in Lemma~\ref{bu-p-toral}. Considering the infinite Stiefel variety $V_\infty^k$, we have 
$$
G_\infty^k = V_\infty^k/O(k),
$$
in other words $V_\infty^k$ is a realization of $EO(k)$. Now let us decompose $\mathbb R^k$ into $p^\alpha$ $2$-dimensional spaces $L_1\oplus\dots\oplus L_{p^\alpha}$. Let the $p^\alpha$-dimensional torus $T$ act on $\mathbb R^k$ by independent rotations of $L_i$. Let the group $F=\mathbb Z_{p^\alpha}$ cyclically permute the spaces $L_i$. In this case we obtain an action of the $p$-toral group $G=T\rtimes F$ on $\mathbb R^k$.

Now consider the section $s_i$ of $\Sigma^2(\gamma_\infty^k)$ as an $O(k)$-equivariant map $f_i : V_\infty^k\to \Sigma^2(\mathbb R^k)$. Restricting the group action to $G$ we see that the product map $f=(f_1,\ldots,f_m)$ is a $G$-equivariant map to a linear representation space. Hence $f$ should map some frame $x\in V_\infty^k$ to an array of $m$ quadratic forms on $\mathbb R^k$, all being $G$-invariant. Now it remains to note that $T$-invariant quadratic forms on $\mathbb R^k$ should be polynomials in $x_1^2+x_2^2, \ldots, x_{k-1}^2 + x_{k}^2$, i.e. they have the form
$$
Q = \sum_{i=1}^{p^\alpha} a_i(x_{2i-1}^2 + x_{2i}^2).
$$
And if we require them to be $F$-invariant, then we obtain the equality
$$
a_1=a_2=\dots=a_{p^\alpha},
$$
which means that $Q$ is proportional to the standard quadratic form. This completes the proof of existence of $n(2, 2p^\alpha, m)$.

Now consider the particular case $k=4$ in Theorem~\ref{grass-quad}, and suppose that the Grassmannians are oriented in the reasonings below, this is needed to apply the results of Section~\ref{rational-sect}. First, consider a simpler case of the bundle $\Sigma^2(\gamma_\infty^4)$ instead of $\Sigma^2(\gamma_{12}^4)$. The cohomology $H^*(G_\infty^4, \mathbb Q)$ (see~\cite{mishch1998}) is a subalgebra of $\mathbb Q[a, b]$ ($a$ and $b$ are two-dimensional generators), generated by $ab$ (the Euler class) and $a^2+b^2$ (the $4$-dimensional Pontryagin class). Since there exists an $SO(4)$-invariant quadratic form, we can decompose the bundle of quadratic forms
$$
\Sigma^2(\gamma_\infty^4) = \xi\oplus\varepsilon.
$$
Now we have to find an obstruction to a nonzero section of the $9$-dimensional bundle $\xi$. From the standard calculation it follows that
$$
p_{16}(\xi) = p_{16}(\Sigma^2(\gamma_\infty^4)) = 16a^2b^2(a^2-b^2)^2.
$$
Hence $\sqrt{p_{16}(\xi)}=4ab(a^2-b^2)$, which does not belong to $H^*(G_\infty^4,\mathbb Q)$. If we consider $G_{12}^4$ instead of the infinite Grassmannian, we see that the kernel of the natural map
$$
H^*(G_\infty^4,\mathbb Q)\to H^*(G_{12}^4,\mathbb Q)
$$
is generated by the dual Pontryagin classes of $\gamma_\infty^4$ (i.e. the Pontryagin classes of its complementary bundle $\gamma_{12}^8$) of dimension $\ge 20$. Such relations do not affect taking a square root of $p_{16}(\xi)$ by the dimension considerations. Besides the image of $H^*(G_\infty^4,\mathbb Q)$ the cohomology $H^*(G_{12}^4,\mathbb Q)$ has another generator: the Euler class of the complementary bundle $c = e(\gamma_{12}^8)$ of dimension $8$, along with the relations 
$$
abc=0,\quad c^2 = (a^2+b^2)^4.
$$
Let us see whether it helps to take a square root of $p_{16}(\xi)$. Assume that (the expression to the left is an arbitrary cohomology class in $H^8(G_{12}^4, \mathbb Q)$)
$$
(xa^2b^2 + y ab(a^2+b^2) + z (a^2+b^2)^2 + tc)^2 = 16a^2b^2(a^2-b^2)^2.
$$
We have a summand $ztc(a^2+b^2)^2$ on the left hand side, hence $z=0$ (we have already shown that $t$ cannot be zero). Then we have
$$
x^2a^4b^4 + 2xy a^3b^3(a^2+b^2) + y^2 a^2b^2 (a^2+b^2)^2 + t^2(a^2+b^2)^4 = 16a^2b^2(a^2-b^2)^2.
$$
Adding another relation $a=b$ we obtain
$$
x^2a^8 + 4xy a^8 + 4 y^2 a^8 + 16 t^2 a^8 = 0,
$$
and therefore 
$$
(x+2y)^2 + 16 t^2 = 0.
$$
The last equality means that $t = 0$, which is already shown to be false.

\begin{rem}
Note that this way of reasoning may work for larger $d$ if we find a $p$-toral subgroup $G\subset O(k)$ which is dense enough in $O(k)$. Unfortunately, by the well-known theorem of Jordan~\cite{jor1878}, for any finite subgroup $G\subset O(k)$ the index of intersection with the maximal torus $[G:G\cap T]$ is bounded by some constant $J(k)$. Hence, for large enough $d$ there exist nontrivial polynomials of degree $d$ in $k$ variables that are invariant under $G$.
\end{rem}

\section{The $2$-Sylow subgroups of the permutation group}
\label{sylow-sec}

We are going to use the $2$-Sylow subgroup $\Sigma_m^{(2)}$ of the permutation group $\Sigma_m$ in the subsequent proofs, let us give a brief explicit description of it. We consider only the case $m=2^\alpha$.

Consider a full graded binary tree $T$ of height $\alpha$ with $m=2^\alpha$ leaves. The group $H=\Sigma_m^{(2)}$ is the group of all automorphisms of $T$ that preserve the grading. It is generated by the involutions that take a non-leaf node $x$ and switch its two children and their corresponding subtrees.

Denote $[m] = \{1,2,\ldots, m\}$. It is obvious that $H$ acts transitively on $[m]$. A pair $(x,y)\in [m]^2$ can be taken to another pair $(z,t)\in [m]^2$ iff the distance between $x$ and $y$ in $T$ is equal to the distance between $z$ and $t$ in $T$.

Consider a more general question: how to describe the orbit of a subset $S\subseteq [m]$ under the action of $H$. Similar to the case of two-element sets, we consider the minimal subtree $T_S\subseteq T$ spanning $S$ and the root of $T$. If $S$ and $R$ are two different subsets of $[m]$, then they belong to the same $H$-orbit iff $T_S$ and $T_R$ belong to the same $H$-orbit. In the trees $T_S$ and $T_R$ the nodes may have $0$, $1$, or $2$ children; call the latter case a \emph{fork}. If $T_S$ and $T_R$ belong to the same $H$-orbit then their respective highest forks should be on the same level of $T$. If it is so, we can identify the highest forks in $T_S$ and $T_R$ by an automorphism from $H$. Under this fork we have two subtrees $T'_S$, $T''_S$ of $T_S$, and two subtrees $T'_R$, $T''_R$ of $T_R$. Then to make an identification of $T_S$ and $T_R$ we have to identify $T'_S$ with $T'_R$ and $T''_S$ with $T''_R$, or $T'_S$ with $T''_R$ and $T''_S$ with $T'_R$. In either case we again search the highest forks in those subtrees, ensure that their gradings are equal, and proceed recursively.

Finally, we note that the orbit of $S$ is defined by the ``topology'' of forks in $T_S$ and their gradings in the tree $T$. This description will be used in Section~\ref{proof-final}.

\section{Proof of Theorem~\ref{ramsey-pol} for $d=4$. The topological part}
\label{d4-proof-top}

Now we have all the prerequisites to prove Theorem~\ref{ramsey-pol}. For the reader's convenience we first outline the proof in the case $d=4$, the final proof is given in Section~\ref{proof-final}. In this particular case, as well as in the general case we combine the topological technique based on the Borsuk--Ulam theorem for $p$-groups with the averaging argument from~\cite{mil1988}.

First, we apply Lemma~\ref{bu-p-toral} and show that it suffices to prove the theorem for a very special type of homogeneous polynomials of degree $4$.

Let $m = 2^\alpha$. Consider the group $G=(\mathbb Z_2)^m\rtimes \Sigma^{(2)}_m$, acting on a $m$-dimensional space $\mathbb R^m$ as follows. Let $(\mathbb Z_2)^m$ act by changing signs of the coordinates, and let $\Sigma^{(2)}_m$ act by permuting the coordinates. 

The group $G$ is obviously a $2$-group and Lemma~\ref{bu-p-toral} tells that if $n$ is large enough, then every homogeneous polynomial of degree $4$ becomes $G$-invariant after restricting to some $m$-dimensional subspace. Now let us describe $G$-invariant polynomials $f$ on $\mathbb R^m$. The invariance w.r.t. $(\mathbb Z_2)^m$ is equivalent to the fact that
$$
f = \sum_{1\le i,j\le m} a_{ij} y_i^2y_j^2,
$$
where $a_{ij}$ is a symmetric $m\times m$ matrix, $y_i$ are the coordinates in $\mathbb R^m$. The $\Sigma^{(2)}_m$-invariance implies more relations on $a_{ij}$, the results in Section~\ref{sylow-sec} give the following description: the number $a_{ij}$ depends only on distance between $i$ and $j$ in the full binary tree $T$, corresponding to the Sylow subgroup $\Sigma_m^{(2)}$. 

\section{Proof of Theorem~\ref{ramsey-pol} for $d=4$. The geometrical part}
\label{d4-proof-geom}

Now we are going to use an averaging argument, similar to what is given in~\cite{mil1988}. Using the remark after the statement of Theorem~\ref{ramsey-pol}, the proof of its particular case for $f=\sum_{i=1}^n x_i^d$ in~\cite{mil1988}, and the result of the above section, we note the following. In order to prove the theorem, we have to find large enough $m=2^\alpha$ for every given $k$ such that
\begin{equation}
\label{solution}
(x_1^2+x_2^2+\ldots+x_k^2)^2 \sim \sum_{1\le i,j\le m} a_{ij} l_i(x)^2l_j(x)^2,
\end{equation}
where $a_{ij}$ is a given $\Sigma^{(2)}_m$-symmetrical matrix, and 
$$
l_i(x) = l_i(x_1, \ldots, x_k)
$$
are some linear forms that we have to find. These forms would give a map $\mathbb R^k\to \mathbb R^m$ such that its image $V$ is the required subspace, because the form $\sum a_{ij} l_i^2l_j^2$ becomes a square after restriction to $V$. Note that these forms have to span $(\mathbb R^k)^*$ to give a map with zero kernel.

We are going to find the forms $l_i(x)$ using the following procedure. Let $s$ be the least power of two that is greater or equal to $k$, let $m = m's$. Choose $s$ linear forms $\lambda_1(x),\ldots, \lambda_s(x)$, with the only restriction that
\begin{equation}
\label{sq-cond}
x_1^2+\ldots+x_k^2 \sim \lambda_1(x)^2+\ldots + \lambda_s(x)^2.
\end{equation}

In this case these $s$ forms already span $(\mathbb R^k)^*$. Let us partition all the forms $l_i$ ($i=1, \ldots, m's$) into consecutive $s$-tuples, and let the $s$-tuple number $t=1,\ldots, m'$ be obtained from the $(\lambda_1, \ldots, \lambda_s)$ by a transform $\sigma_t\in SO(k)\times \mathbb R^+$ (a rotation with a positive homothety), i.e.
$$
l_{st-i}(x) = \lambda_{s-i}(\sigma_tx),
$$
where $i=0,\ldots, s-1$. Every $\sigma_t$ multiplies the quadratic form $x_1^2+\ldots + x_k^2$ by a positive number, hence Equation~\ref{sq-cond} holds for every considered $s$-tuple of $l_i$'s.

Note that the right hand part of Equation~\ref{solution} can be rewritten using Equation~\ref{sq-cond} (and the symmetry of $a_{ij}$) as follows
$$
(x_1^2+x_2^2+\ldots+x_k^2)^2 \sim \left(\sum_{t=1}^{m'}\sum_{1\le i,j\le s} a_{ij} \lambda_i(\sigma_tx)^2\lambda_j(\sigma_tx)^2\right) + B(x_1^2+\ldots + x_k^2)^2.
$$
The first summand is formed by $s\times s$ cells on the diagonal of $a_{ij}$. Each of the non-diagonal $s\times s$ cells of $a_{ij}$ consists of a single constant (from the $\Sigma^{(2)}_m$-symmetry condition). Hence, the non-diagonal $s\times s$ cells give a summand proportional to $(x_1^2+\ldots + x_k^2)^2$ (from Equation~\ref{sq-cond}).

Now denote
$$
g(x) = \sum_{1\le i,j\le s} a_{ij} \lambda_i(x)^2\lambda_j(x)^2,
$$
we have to prove that for some $\sigma_1,\ldots, \sigma_{m'}\in SO(k)\times \mathbb R^+$
\begin{equation}
\label{average}
(x_1^2+x_2^2+\ldots+x_k^2)^2 \sim \sum_{t=1}^{m'} g(\sigma_tx).
\end{equation}
The rest of the proof is similar to the proof in~\cite{mil1988}. If we substitute the right hand part of Equation~\ref{average} by an integral over every possible rotation $\rho\in SO(k)$, we surely obtain an $SO(k)$-invariant $4$-form, which has to be proportional to $(x_1^2+\ldots+x_k^2)^2$. Then we use the Carath\'eodory theorem to show that if 
$$
m'\ge \binom{k+3}{4},
$$
then some $m'$ rotations give a symmetric convex combination 
$$
\sum_{t=1}^{m'} w_tg(\rho_tx) \sim (x_1^2+x_2^2+\ldots+x_k^2)^2,
$$
now it suffices to denote $\sigma_t = \rho_t\sqrt[4]{w_t\phantom{|}}$ and use the fact that $g(x)$ is $4$-homogeneous.

\begin{rem} The total estimate on $n(4, k)$ in this proof is not very good. The averaging argument gives $m\sim k^5$, after that $n$ is determined by $m$ using the Borsuk--Ulam property. The latter estimate is not known directly, because it uses asymptotic facts on the equivariant cohomotopy of classifying spaces. From a detailed analysis of the proof it is clear that the group $\Sigma^{(2)}_m$ can be replaced by a smaller subgroup, but anyway, the group depends on $k$ and the estimate on $n(m)$ is not known. See also the discussion at the end of Ch.~3 in~\cite{bart1993}, the results conjectured there would imply a polynomial bound for $n(m)$.
\end{rem}

\section{Proof of Theorem~\ref{ramsey-pol} for arbitrary $d$}
\label{proof-final} 

First, let us apply the Borsuk--Ulam theorem for $2$-groups to the group $G=(\mathbb Z_2)^m\rtimes \Sigma^{(2)}_m$, $m$ is to be defined later. If the initial dimension $n$ is large enough, then the restriction of $f$ to some $m$-dimensional subspace equals (put $d/2 = \delta$)
\begin{equation}
\label{2gr-symm}
f = \sum_{1\le i_1,i_2,\ldots, i_\delta\le m} a_{i_1,\ldots,i_\delta} y_{i_1}^2y_{i_2}^2\dots y_{i_\delta}^2,
\end{equation}
where the numbers $a_{i_1,\ldots,i_\delta}$ are invariant under the component-wise action of $\Sigma^{(2)}_m$ on the indexes $i_1,\ldots, i_\delta$. 

We are going to use the averaging argument from~\cite{mil1988} several times and for several polynomials simultaneously, so we describe the averaging procedure in a single lemma. Denote the group of rotations composed with a homothety by $S(k) = SO(k)\times \mathbb R^+$, call its elements \emph{similarity transforms}. Sometimes we consider the zero map as a similarity transform too, denote $S_0(k)= S(k)\cup\{0\}$.

\begin{lem}
\label{avg-pol}
Suppose $f_1, \ldots, f_l$ are even homogeneous polynomials of degree $\le d$ in $k$ variables,
$$
n\ge l \binom{k+d-1}{d}.
$$
We can find $n$ similarity transforms $\sigma_1,\ldots,\sigma_\in S_0(k)$ (not all zero) such that all the polynomials 
$$
\overline f_j(x) = \sum_{i=1}^n f_j(\sigma_i x)
$$
are proportional to $Q^{\deg f_j/2} = (x_1^2+\ldots+x_k^2)^{\deg f_j/2}$.
\end{lem}

\begin{proof}
Note that if a polynomial $f_j(x)$ has degree $d'<d$, we can multiply it by $Q^{\frac{d-d'}{2}}$ and assume that all $f_j(x)$ has the same degree $d$.

As in~\cite{mil1988}, the polynomials 
$$
\overline f_j(x) = \int_{\rho\in SO(k)} f_j(\rho x) d\rho
$$
are $SO(k)$-invariant, and therefore proportional to $Q^{d/2}$. Note that the linear space of $l$-tuples of polynomials of degree $d$ in $k$ variables has dimension $l \binom{k+d-1}{d}$, and by the Carath\'eodory theorem the vector $(\overline f_1,\ldots, \overline f_l)$ is proportional to a convex combination 
$$
(\overline f_1(x),\ldots, \overline f_l(x)) = \sum_{i=1}^n w_i (f_1(\rho_i x),\ldots, \overline f_l(\rho_i x))
$$
for some $\rho_1,\ldots,\rho_n\in SO(k)$. Putting $\sigma_i = w_i^{1/d}\rho_i\in S_0(k)$, we obtain the required formula.
\end{proof}

\begin{rem}
Note that the averaging procedure is linear in the polynomials $f_j$. We can take $l=\binom{k+d-1}{d}$ and average \emph{all} the even polynomials of degree $\le d$ by the same sequence $\sigma_1,\ldots,\sigma_n\in S_0(k)$ for $n=\binom{k+d-1}{d}^2$.
\end{rem}

Let us describe the structure of a $G$-invariant polynomial $f$ in Equation~\ref{2gr-symm} in more detail. Denote $H=\Sigma^{(2)}_m$ for brevity. Similar to what is done in Section~\ref{sylow-sec}, let us describe the $H$-orbits of the multisets $(i_1,\ldots, i_\delta)$. The coefficients $a_{i_1,\ldots, i_\delta}$ of $f$, corresponding to the same orbit, should be equal. Let us identify the set $[m]$ with the leaves of the full binary tree $T$ of height $h=\log_2 m$. Again, the group $H$ is the group of automorphisms of $T$ preserving the grading. We choose the grading so that the leaves are of grading zero, what is above them is of grading $1$, and so on.

For every multiset $S=(i_1,\ldots, i_\delta)$ consider the subtree $T_S$ spanning $S$ and the root of $T$. The orbit of $S$ is fully characterized by the corresponding orbit of $T_S$ with assigned multiplicities (of the multiset $S$) to the leaves of $T_S$. As in Section~\ref{sylow-sec} the orbit of $T_S$ is fully described by gradings of its forks, the parent-child relation between them, and the multiplicities of leaves. In this description the only value that depends on $h$ is the gradings, therefore the number of $H$-orbits of trees $T_S$ with multiplicities on leaves (and the number of $H$-orbits of index multisets $S$) is
$$
\le C(\delta) h^{\delta-1}.
$$

In the sequel we will use the above tree description for different heights $h$, so denote $H_h=\Sigma^{(2)}_{2^h}$ the $2$-Sylow permutation group, and the corresponding full binary tree $T_h$. Suppose $U$ is a multiset of cardinality $\delta'\le \delta$ in $[2^h]$ (the leaves of $T_h$), denote by
$$
g_U(y_1,\ldots, y_{2^h}) = \sum_{\substack{\sigma\in H_h\\\sigma(U) = (i_1,\ldots,i_{\delta'})}} y_{i_1}^2\dots y_{i_{\delta'}}^2
$$
the $H_h$-invariant polynomials. The number of such distinct polynomials $g_U$ is $\le C(\delta) h^\delta$, and the considered polynomial $f$ is a linear combination of such polynomials in $m$ variables for $\#U=\delta$ (we denote $\#U$ the cardinality of the multiset $U$).

We are going to find $m$ linear forms $l_t(x)$ on $\mathbb R^k$ such that every polynomial $g_U(l_1(x),\ldots,l_m(x))$ (after substituting $y_t=l_t(x)$) is proportional to $Q^\delta$, this will imply the same claim about $f(l_1(x),\ldots,l_m(x))$ by linearity.

Take some nonzero linear function $l_1(x_1,\ldots, x_k)$ in $k$ variables. Then build the other forms $l_i(x)$ by the following procedure. Define the sequence of the powers of two $s_0 = 1, s_1=2^{h_1},\ldots, s_\delta = 2^{h_\delta}$ that satisfies the following inequality
$$
s_{i+1} \ge C(\delta) {h_i}^\delta \binom{k+d-1}{d} s_i.
$$
Then suppose we have already chosen the linear functions $l_1(x),\ldots,l_{s_i}(x)$. Consider all the polynomials $g_U(y_1,\ldots, y_{s_i})$, corresponding to multisets $U$ in $[s_i]$ of cardinality $\le \delta$, the number of distinct such polynomials is at most $C(\delta) {h_i}^\delta$. Put 
$$
\phi_U(x) = g_U(l_1(x), \ldots, l_{s_i}(x))
$$
and apply Lemma~\ref{avg-pol} to the polynomials $\phi_U(x)$ to obtain $n\le C(\delta) {h_i}^\delta\binom{k+d-1}{d}$ similarity transforms $\sigma_1,\ldots,\sigma_n\in S_0(k)$ (not all zero) such that all the expressions
$$
\sum_{j=1}^n g_U(l_1(\sigma_j x), \ldots, l_{s_i}(\sigma_j x))
$$
are proportional to $Q^{\#U}$. Denote for $t = s_i (j-1) + r$ ($1\le j\le n$, $1\le r \le s_i$)
$$
l_t(x) = l_r(\sigma_j x).
$$
Now we have $s_{i+1} = ns_i$ linear functions, consisting of $n$ similar copies of the previous set of $s_i$ linear functions.

Finally we define $m = s_\delta$, note that here $m$ is roughly of order $C(d) (k\log k)^{d^2/2}$. We could also take $m \le 2k\binom{k+d-1}{d}^d$ using the remark after Lemma~\ref{avg-pol}, this bound is worse but does not contain unknown functions of $d$. Note again that the explicit bound in this theorem depends on the (unknown) explicit bound on $n$ in terms of $m$ and $k$ in the Borsuk--Ulam theorem for $2$-groups.

Consider a polynomial $g_S$, corresponding to a multiset $S$ in $[m]$ of cardinality $\delta$. Denote for brevity for a multiset $S=(i_1,\ldots,i_\delta)$
$$
\overline y^{2S} = (y_1,\ldots, y_m)^{2S} =y_{i_1}^2y_{i_2}^2\dots y_{i_\delta}^2.
$$
The corresponding tree $T_S$ has no forks with gradings in $(h_i,h_{i+1}]$ for some $i$ by the pigeonhole principle (it has $\le \delta-1$ forks). If we fix the part $T_0$ of this tree with gradings $>h_{i+1}$, and cut it off, then we obtain several subtrees $T_1,\ldots, T_r$ of height $h_{i+1}$, with forks no higher than $h_i$, denote their corresponding leaf multisets $S_1,\ldots, S_r$. If we decompose $[m]$ into segments $I_1,\ldots, I_{m/s_{i+1}}$ of length $s_{i+1}$ each, then we see that the leaf multisets $S_j$ are intersections of $S$ with the corresponding segments $I_j$. Consider the orbit of $S$ under the group 
$$
F_1\times \dots\times F_{m/s_{i+1}} \subseteq H_h,
$$
where $F_j=\Sigma^{(2)}_{s_{i+1}}$ is the group of $2$-Sylow permutations in every $I_j$. Denote $r=m/s_{i+1}$ The sum of the corresponding monomials
$$
\sum_{\gamma_1\times\dots\times\gamma_r\in F_1\times \dots\times F_r} \overline y^{2\gamma_1\times\dots\times\gamma_r(S)}
$$
can be rewritten as the product
$$
\prod_{\substack{1\le j\le r\\ S\cap I_j\neq\emptyset}} \sum_{\gamma_j\in F_j} \overline y^{2\gamma_j(S\cap I_j)}.
$$
Every expression $\sum_{\gamma_j\in F_j} \overline y^{2\gamma_j(S_j)}$ (we denote $S_j=I_j\cap S$), corresponding to a particular $I_j$, is a homogeneous polynomial of the form $g_{S_j}$ in $s_{i+1}$ variables $y_t$ ($t\in I_j$). Consider the subgroup $G_j\in F_j$, consisting of elements of the form $\alpha\times\dots\times\alpha$, where $\alpha\in \Sigma^{(2)}_{s_i}$ is a permutation of size $[s_i]$, i.e. $G_j$ permutes all the $s_i$-blocks of $I_j$ in the same way. Consider also the subgroup $K\subset F_j$ (isomorhic to $(\mathbb Z_2)^{h_{i+1}/h_i}$) that permutes the whole $s_i$-blocks in $I_j$ transitively, this is the group, generated by applying the same transposition of children at a particular binary tree level. Here we assume that any group $(\mathbb Z)^\alpha$ permutes a set of cardinality $2^\alpha$ by an action isomorphic to the left action of $(\mathbb Z)^\alpha$ on itself. Note that these two groups generate a Cartesian product subgroup $G_j\times K\subseteq F_j$.

Note that the corresponding to $S_j$ tree $T_j$ has no forks higher than $h_i$. The sum 
\begin{equation}
\label{sum-sum}
\sum_{\gamma\times \kappa\in G_j\times K} \overline y^{2\gamma\times \kappa(S_j)}
\end{equation}
can be rewritten as summation over $\gamma\in G_j$, and then on $\kappa\in K$. The first summation gives a polynomial of type $g_{S_j}$ in $s_i$ variables $y_t$, corresponding to the $s_i$-block of $I_j$, where all elements $S_j$ are contained (because $T_j$ has no forks higher than $h_i$ and $S_j$ is contained in a single $s_i$-block). 

If we substitute $y_t = l_t(x)$ and sum such polynomials $g_{S_j}$ over $K$, we obtain an expression in $x_1,\ldots, x_k$ proportional to $Q^{\#S_j}$ by the construction of the linear functions $l_t(x)$. Suppose $I_j$ is the segment $I_1$ of the first $s_{i+1}$ variables $y_t$, the first $s_i$ functions $l_t(x)$ of this segment were transformed into $s_{i+1}$ functions $l_t(\sigma_u(x))$ by the construction, so that the summation over $u=1,\ldots, s_{i+1}/s_i$ of the expressions  
$$
g_{S_j}(l_1(\sigma_u x),\ldots, l_{s_i}(\sigma_u x))
$$ 
makes them proportional to $Q^{\#S_j}$. The same holds for every (not only the first) $s_{i+1}$-segment $I_j$, because all the corresponding linear functions $\{l_t(x) : t\in I_j\}$ are obtained from the linear functions $\{l_t(x) : t\in I_1\}$ by substituting $l_t(\tau x)$ with the same similarity transform $\tau$, which appears in the construction of $s_{i+2},\ldots, s_\delta$. In this case we make the summation of  
$$
g_{S_j}(l_1(\sigma_u\tau x),\ldots, l_{s_i}(\sigma_u\tau x))
$$ 
over $u=1,\ldots, s_{i+1}/s_i$ (i.e. over the group $K$), and by the construction we obtain a polynomial proportional to $Q(\tau x)^{\#S_j}$, which is proportional to $Q(x)^{\#S_j}$, since $\tau$ is a similarity transform.

If we pass to summation in Equation~\ref{sum-sum} over a larger group $F_j\supseteq G_j\times K$, then we again obtain a sum similar to $Q^{\#S_j}$ after substituting $y_t = l_t(x)$. The same argument is valid for summation of the monomials $\overline y^{2S}$ over the entire group $H_h\supseteq F_1\times \dots\times F_r$, so every such sum will be proportional to $Q^\delta$ after substituting $y_t=l_t(x)$, as required.

Note that the case $i+1=1$ is done in the same manner assuming $s_0=1$.

\begin{rem}
Note that if we use the ``universal averaging'', according to the remark after Lemma~\ref{avg-pol}, then the only needed averaging in the proofs is averaging over the groups $K_i$ (denoted simply $K$ in the proof). These groups are $2$-tori of size $2^{h_{i+1}/h_i}$, and the total required group is the wreath product $K_1\wr K_2\wr\dots\wr K_\delta$. This observation may help if some explicit bounds in the Borsuk--Ulam theorem for wreath products of $2$-tori are found.
\end{rem}

\section{Sections and projections of convex bodies}
\label{proj-sec}

Now we return to the case of quadratic forms, and deduce some corollaries for sections or projections of convex bodies from Theorems~\ref{grass-quad} and \ref{grass-quad-mult}. We use the standard approach (e.g. see~\cite{ziv2004}) of taking geometric consequences of the results over Grassmanians.

\begin{cor}
Suppose $k$ is an integer of the form $2p^\alpha$, $m\ge 1$, $n\ge n(2, k, m)$ from Theorem~\ref{grass-quad-mult} or \ref{grass-quad}. Let $K_1,\ldots, K_m$ be convex bodies in $\mathbb R^n$. Then there exists a $k$-dimensional linear subspace $L\subseteq \mathbb R^n$ such that the orthogonal projections of any $K_i$ onto $L$ has a Euclidean ball as its John ellipsoid.
\end{cor}

\begin{proof}
Consider all possible choices of $L$, they form the Grassmannian $G_n^k$. The John ellipsoid~\cite{john1948} of the projection $\pi_L(K_i)$ depends continuously on $L$, its homogeneous component of degree $2$ is a quadratic form on $L$, hence it gives a section $s_i$ of $\Sigma^2(\gamma_n^k)$. By Theorem~\ref{grass-quad-mult} these quadratic forms are simultaneously proportional to the standard quadratic form over some $L$.
\end{proof}

The following corollary is proved in the same way.

\begin{cor}
Suppose $k$ is an integer of the form $2p^\alpha$, $m\ge 1$, $n\ge n(2, k, m)$ from Theorem~\ref{grass-quad-mult} or \ref{grass-quad}. Let $K_1,\ldots, K_m$ be convex bodies in $\mathbb R^n$, and $x$ be a point in the interior of $\bigcap_{i=1}^m K_i$. Then there exists a $k$-dimensional affine subspace $x\in L\subseteq \mathbb R^n$ such that for any $i$ the section $K_i\cap L$ has a Euclidean ball as its John ellipsoid.
\end{cor}

It is easy to see that instead of the John ellipsoid we can consider the second moment matrix of the projection (or the section), or some other quadratic form, depending continuously on the convex body. Note that some ``approximate'' version of these theorems follows form the original Dvoretzky theorem, e.g. we can state that the John ellipsoid is $\varepsilon$-close to a ball. 

\section{The weak form of the Knaster conjecture}

Let us state the weak form of the Knaster conjecture from~\cite{kna1947}.

\begin{con}
There exists $n=n(l)$ such that for any $l$ points $X=\{x_1,\ldots,x_l\}$ on the unit sphere $S^{n-1}$ and any continuous function $f : S^{n-1}\to \mathbb R$ there exists a rotation $\rho\in O(n)$ such that
$$
f(\rho x_1) = f(\rho x_2) = \dots = f(\rho x_l).
$$
\end{con}

Originally Knaster conjectured that $n(l)=l$, but counterexamples to his conjecture were found in~\cite{kasha2003}. In~\cite{flo1955} it was proved that $n(3)=3$, but already the value $n(4)$ is not known and not shown to be finite. Known results in this conjecture either consider sets $X$, distributed along a two-dimensional vector subspace of $\mathbb R^n$ (see~\cite{mil1988,mak1990}), or require very specific symmetry conditions, e.g. require $X$ to be an (almost) orthogonal frame (see~\cite{vol1992umn}).

In~\cite{mil1988} it was shown that the original Knaster conjecture ($n(l)=l$) would imply the Dvoretzky theorem with good estimates on $n(k, \epsilon)$, it would also imply Theorem~\ref{ramsey-pol}. The weak form of the Knaster conjecture would also give some bounds in the Dvoretzky theorem, as well as explicit bounds in Theorem~\ref{ramsey-pol}. In order to prove Dvoretzky type results we have to consider sets $X$ distributed densely enough in a sphere $S^{k-1}$ of given dimension.


\begin{thebibliography}{99}

\bibitem{arha2006}
R.M.~Aron, P.~H\'ajek. Zero sets of polynomials in several variables. // Archiv der Mathematik, 86, 2006, 561--568.

\bibitem{bcp1991}
T.~Bartsch, M.~Clapp, D.~Puppe. A mountain pass theorem for actions of compact Lie groups. // J. reine angew. Math. 419, 1991, 55--66.

\bibitem{bart1992}
T.~Bartsch. On the existence of Borsuk--Ulam theorems. // Topology, 31, 1992, 533--543.

\bibitem{bart1993}
T.~Bartsch. Topological methods for variational problems with symmetries. Berlin-Heidelberg: Springer-Verlag, 1993.

\bibitem{bir1957}
B.J.~Birch. Homogeneous forms of odd degree in a large number of variables. // Mathematika, 4, 1957, 102--105.

\bibitem{buivta2009}
D.~Burago, S.~Ivanov, S.~Tabachnikov. Topological aspects of the Dvoretzky theorem. // \href{http://arxiv.org/abs/0907.5041}{arXiv:0907.5041v1}, 2009; to appear in Journal of Topology and Analysis.

\bibitem{carl1984}
G.~Carlsson. Equivariant stable homotopy and Segal's Burnside ring conjecture. // Ann. of Math., 120,
1984, 189--224.

\bibitem{clm2000}
M.~Clapp, W.~Marzantowicz. Essential equivariant maps and Borsuk--Ulam theorems. // J. London Math. Soc., 61(2), 2000, 950--960.

\bibitem{dvor1961}
A.~Dvoretzky. Some results on convex bodies and Banach spaces. // Proc. International
Symposium on Linear spaces, Jerusalem, 1961, 123--160.

\bibitem{flo1955}
E.E.~Floyd. Real valued mappings of spheres. // Proc. Amer. Math. Soc., 6, 1955, 1957--1959.

\bibitem{hatcher2002}
A.~Hatcher. Algebraic Topology. Cambridge University Press, 2002.

\bibitem{hsiang1967}
W.C.~Hsiang, W.Y.~Hsiang. Differentiable actions of compact connected classical group I. // Amer. J. Math., 89, 1967, 705--786.

\bibitem{hsiang1975}
W.Y.~Hsiang. Cohomology theory of topological transformation groups. Springer Verlag, 1975.

\bibitem{john1948}
F.~John. Extremum problems with inequalities as subsidiary conditions. // Studies and Essays Presented to R.~Courant on his 60th Birthday, January 8, 1948, 187—-204.

\bibitem{jor1878}
C.~Jordan. M\'emoire sur les \'equations differentielles lin\'eaires \'a int\'egrale
alg\'ebrique, J. f\"ur Math. 84, 1878, 89--215.

\bibitem{kasha2003}
B.S.~Kashin, S.J.~Szarek. The Knaster problem and the geometry of high-dimensional cubes. // Comptes Rendus Mathematique,
336(11), 2003, 931--936.

\bibitem{kna1947}
B.~Knaster. Problem 4. // Colloq. Math., 30, 1947, 30--31.

\bibitem{mishch1998}
G.~Luke, A.S.~Mishchenko. Vector bundles and their applications. Springer Verlag, 1998.


\bibitem{mak1990}
V.V.~Makeev. The Knaster problem and almost spherical sections. // Mathematics of the USSR--Sbornik, 66(2), 1990, 431--438.

\bibitem{mak2003}
V.V.~Makeev. Universally inscribed and outscribed polytopes. Doctor of mathematics thesis. Saint-Petersburg State University, 2003.

\bibitem{mil1988}
V.D.~Milman. A few observations on the connections between local theory and some other fields. // Geometric Aspects of Functional Analysis, Lecture Notes in Mathematics 1317, 1988, 283-289.

\bibitem{mil1992}
V.D.~Milman. Dvoretzky's theorem -- thirty years later. // Geometric and Functional Analysis, 2(4), 1992, 455-479.

\bibitem{milsta1974}
J.~Milnor, J.~Stasheff. Characteristic classes. Princeton University Press, 1974.


\bibitem{sch1913}
E.~Schmidt, Zum Hilbertschen Baveise des Waringschen Theorems. // Math. Ann., 77, 1913, 271--274.

\bibitem{vol1992umn}
A.Yu.~Volovikov. On maps of Stiefel manifolds with a free $ \mathbb Z^n_p$-action in the manifold (In Russian). // Uspehi Mat. Nauk, 47:6(288), 1992, 205--206; translation in Russian Math. Surveys, 47:6, 1992, 235--236.

\bibitem{ziv2004}
R.~\v Zivaljevi\'c. Topological methods. // Handbook of Discrete and Computational Geometry, ed. by J.E.~Goodman, J.~O'Rourke, CRC, Boca Raton, 2004.

\end{thebibliography}
\end{document}